\documentclass[11pt]{article}
\usepackage[margin=1.05in]{geometry}
\usepackage{amsmath,amssymb,amsthm,mathtools}
\usepackage{microtype}
\usepackage{tikz}
\usetikzlibrary{arrows.meta}
\usepackage[hidelinks]{hyperref}
\newtheorem{theorem}{Theorem}[section]
\newtheorem{proposition}[theorem]{Proposition}
\newtheorem{lemma}[theorem]{Lemma}
\newtheorem{corollary}[theorem]{Corollary}
\theoremstyle{remark}

\newcommand{\R}{\mathbb R}
\newcommand{\Z}{\mathbb Z}
\newcommand{\E}{\mathbb E}
\newcommand{\Prob}{\mathbb P}
\newcommand{\capD}{\operatorname{cap}_{\Lambda}(D)}
\newcommand{\dist}{\operatorname{dist}}
\newcommand{\M}{\mathcal M}
\title{Macroscopic freezing of a vector Gaussian free field\\ conditioned to avoid a ball}
\author{Yan Ru Pei}
\date{September 14, 2026}
\hypersetup{pdftitle={Macroscopic freezing of a vector Gaussian free field conditioned to avoid a ball},pdfauthor={Yan Ru Pei}}
\begin{document}
\maketitle
\begin{abstract}
We consider a fixed finite number of independent two-dimensional discrete
Gaussian free fields on a box of side $n$, with zero boundary values,
conditioned so that the vector-valued field stays outside a fixed ball
on a macroscopic interior domain.
We prove that the field, divided by $\log n$, converges in distribution in
$L^2$ to twice the equilibrium potential of the domain multiplied by a
uniform random unit vector. We also identify the exact capacity cost of the
conditioning event and obtain alignment on a deterministic density-one set
of sites, with no restriction on their mutual distance.
The proof combines the norm repulsion theorem of
Korzhenkova and Sep\'ulveda with Gaussian large deviations and the equality
case in the capacity variational problem.
\end{abstract}

\section{Introduction and result}

Entropic repulsion describes the displacement of a random surface forced
to avoid an obstacle. For a vector-valued field, avoiding a ball leaves the
direction of that displacement free. The resulting question is whether
different parts of a macroscopic domain choose the same direction.
Korzhenkova and Sep\'ulveda~\cite{KS} proved that the norm of a
two-dimensional vector Gaussian free field is typically $2\log n$ under
this conditioning and that its directions align at distances up to $n^\nu$
for every fixed $\nu<1$. Their Remark~6 asks for macroscopic freezing.
We establish a full spatial profile, which in particular gives this
alignment.

Let $\Lambda=(-1/2,1/2)^2$, and let $D$ be a nonempty open set with smooth
boundary, $0\in D$, and $\overline D\subset\Lambda$.
Write $\Lambda_n=[-n/2,n/2]^2\cap\Z^2$ and $D_n=nD\cap\Z^2$.
Fix an integer $m\ge2$ and $R>0$. Let
$\varphi_n=(\varphi_n^1,\ldots,\varphi_n^m)$ have independent
zero-boundary Gaussian free field coordinates on $\Lambda_n$.
Our normalization is the density
\begin{equation}\label{eq:density}
 \frac1{Z_n}\exp\left\{-\frac1{4\pi}
       \sum_{\{x,y\}:\,x\sim y,\ \{x,y\}\cap\Lambda_n\ne\varnothing}
       |\varphi_n(x)-\varphi_n(y)|^2\right\},
\end{equation}
where each edge is counted once and the field is zero off $\Lambda_n$.
Thus each coordinate has covariance $2\pi L_n^{-1}$, where $L_n$
is the Dirichlet graph Laplacian, and its variance at the origin is
asymptotic to $\log n$.
Set
\[
 \Omega_n=\{|\varphi_n(x)|\ge R\text{ for all }x\in D_n\},
 \qquad Q_n=\Prob(\,\cdot\mid\Omega_n).
\]
We embed $\varphi_n/\log n$ as a step function $X_n$ on $\Lambda$:
on the cell $(x+[-1/2,1/2)^2)/n$, intersected with $\Lambda$,
its value is $\varphi_n(x)/\log n$.

Define the relative capacity and its equilibrium potential by
\begin{equation}\label{eq:capacity}
 \capD=\frac12\inf\left\{\int_\Lambda|\nabla h|^2:
       h\in H_0^1(\Lambda),\ h\ge1\text{ a.e. on }D\right\},
 \qquad h_D=\mathop{\rm argmin}(\cdots).
\end{equation}
The minimizer is unique, equals one on $D$, and is strictly positive
throughout $\Lambda$. In particular, it couples the directions chosen in
different components even if $D$ is disconnected.

\begin{theorem}\label{thm:main}
With the preceding assumptions and normalization,
\begin{equation}\label{eq:cost}
 \lim_{n\to\infty}\frac{\log\Prob(\Omega_n)}{(\log n)^2}
       =-\frac2\pi\capD.
\end{equation}
Under $Q_n$, the random functions $X_n$ converge in distribution in
$L^2(\Lambda;\R^m)$ to $2Uh_D$, where $U$ is uniform on $S^{m-1}$.
In particular, if
$\M=\{2eh_D:e\in S^{m-1}\}$, then for every $\delta>0$,
\begin{equation}\label{eq:shape}
 Q_n\bigl(\dist_{L^2}(X_n,\M)\ge\delta\bigr)\longrightarrow0.
\end{equation}
\end{theorem}

\begin{corollary}\label{cor:spins}
There are deterministic sets $\widetilde D_n\subset D_n$ with
$|D_n\setminus\widetilde D_n|=o(n^2)$ such that, writing
$\sigma_n(x)=\varphi_n(x)/|\varphi_n(x)|$ on $\Omega_n$,
\[
 \inf_{x,y\in\widetilde D_n}
       \E[\sigma_n(x)\cdot\sigma_n(y)\mid\Omega_n]\longrightarrow1.
\]
\end{corollary}

The crucial point is that norm repulsion yields the limiting constraint
$|f|\ge2$ almost everywhere on $D$, with violations measured by a bounded
functional continuous in $L^2$. The Gaussian rate function then forces the conditioned
field toward the minimizers of a vector capacity problem. Taking the
modulus reduces that problem to scalar capacity; equality in the energy
inequality forces a single direction throughout the box.
This large-deviation and capacity approach has classical scalar
antecedents, including Ben Arous and Deuschel~\cite{BAD} and
Bolthausen, Deuschel and Giacomin~\cite{BDG}.
We give the spectral normalization and approximation details below.

\section{The two inputs from norm repulsion}

We use the following consequences of~\cite{KS}, in the normalization
\eqref{eq:density}. The first is their Theorem~1.1.2, and the second is
their Corollary~4.1.2.
\begin{proposition}[Korzhenkova--Sep\'ulveda]\label{prop:KS}
For every $\eta,\varepsilon>0$,
\begin{equation}\label{eq:norm-repulsion}
 \#\left\{x\in D_n:
 Q_n\left(\left|\frac{|\varphi_n(x)|}{\log n}-2\right|
                    \ge\eta\right)>\varepsilon\right\}=o(n^2).
\end{equation}
Moreover,
\begin{equation}\label{eq:lower-cost}
 \liminf_{n\to\infty}\frac{\log\Prob(\Omega_n)}{(\log n)^2}
                 \ge-\frac2\pi\capD .
\end{equation}
\end{proposition}

Only a bounded lower deficit is needed to use this input; no conditional
moment estimate is required. Define
\[
 G(f)=\int_D(2-|f(u)|)_+^2\,du,\qquad f\in L^2(\Lambda;\R^m).
\]
This is a bounded continuous functional. Indeed, the function
$z\mapsto(2-|z|)_+^2$ is bounded by $4$ and is $4$-Lipschitz.

\begin{lemma}\label{lem:deficit}
$\E_{Q_n}G(X_n)\to0$.
\end{lemma}
\begin{proof}
The union of the cells with centers in $D_n/n$ differs from $D$
by area $O(n^{-1})$. Thus its use in $G$ changes the answer by
$O(n^{-1})$. For fixed $\eta,\varepsilon>0$, discard the $o(n^2)$
exceptional sites in \eqref{eq:norm-repulsion}. At every remaining site,
the expected squared lower deficit is at most $\eta^2+4\varepsilon$.
It follows that
\[
 \limsup_n\E_{Q_n}G(X_n)\le |D|(\eta^2+4\varepsilon).
\]
Let $\eta,\varepsilon$ decrease to zero.
\end{proof}

\section{The vector capacity problem}

Consider the good rate function
\begin{equation}\label{eq:rate}
 I(f)=
 \begin{cases}
  \displaystyle\frac1{4\pi}\int_\Lambda|\nabla f|^2,
       &f\in H_0^1(\Lambda;\R^m),\\
  +\infty,&\text{otherwise}.
 \end{cases}
\end{equation}
Compact Sobolev embedding makes the sublevel sets of $I$ compact in $L^2$.
Set $C_D=(2/\pi)\capD$.

\begin{lemma}\label{lem:minimizers}
The minimum of $I$ on $\{G=0\}$ is $C_D$, and its set of minimizers
is exactly $\M$.
\end{lemma}
\begin{proof}
For finite-rate $f$, its modulus $r=|f|$ belongs to $H_0^1(\Lambda)$,
with $|\nabla r|\le|\nabla f|$ almost everywhere.
When $G(f)=0$, the function $r/2$ is admissible in
\eqref{eq:capacity}. Hence $I(f)\ge C_D$.
Each $2eh_D$ attains this bound.

If equality holds, uniqueness in the scalar problem gives
$|f|=2h_D$. The equilibrium potential is continuous and bounded away
from zero on every compact subset of $\Lambda$.
Consequently $s=f/(2h_D)$ belongs to $H^1_{\rm loc}(\Lambda;\R^m)$
and $|s|=1$. The Sobolev product rule and
$s\cdot\partial_j s=0$ give
\[
 |\nabla f|^2=4|\nabla h_D|^2+4h_D^2|\nabla s|^2.
\]
Equality of the energies implies $\nabla s=0$ almost everywhere.
Since $\Lambda$ is connected, $s$ is one constant unit vector.
\end{proof}

Figure~\ref{fig:orientation} illustrates why the equality case forces
the same direction even on disconnected components of $D$.

\begin{figure}[tbp]
\centering
\begin{tikzpicture}[
  x=1.35cm,y=1.35cm,
  every node/.style={font=\small},
  fieldarrow/.style={-{Stealth[length=1.7mm,width=1.2mm]},line width=.7pt},
  obstacle/.style={draw=black!45,fill=black!7,line width=.5pt}
]
\begin{scope}
  \node at (1.75,3.03) {(a) Constant direction};
  \draw[line width=.7pt] (0,0) rectangle (3.5,2.6);
  \node[anchor=south west] at (0,2.6) {$\Lambda$};
  \path[obstacle] (1.05,1.10) ellipse (.80 and .55);
  \path[obstacle] (2.62,1.70) ellipse (.42 and .45);
  \node at (1.03,.87) {$D_1$};
  \node at (2.63,1.46) {$D_2$};
  \foreach \x/\y/\len in {
    .35/.40/.30,.48/2.15/.28,.18/1.72/.22,
    .70/1.20/.65,2.30/1.83/.65,1.95/.58/.40,
    1.75/2.18/.30,2.94/.68/.28}
    \draw[fieldarrow,blue!65!black] (\x,\y) -- ++(18:\len);
  \node at (1.75,-.34) {$s(u)=U$};
  \node at (1.75,-.70) {Zero angular energy};
\end{scope}
\begin{scope}[xshift=6.4cm]
  \node at (1.75,3.03) {(b) Varying direction};
  \draw[line width=.7pt] (0,0) rectangle (3.5,2.6);
  \node[anchor=south west] at (0,2.6) {$\Lambda$};
  \path[obstacle] (1.05,1.10) ellipse (.80 and .55);
  \path[obstacle] (2.62,1.70) ellipse (.42 and .45);
  \node at (1.03,.87) {$D_1$};
  \node at (2.63,1.46) {$D_2$};
  \foreach \x/\y/\len/\angle in {
    .35/.40/.30/10,.48/2.15/.28/16,.18/1.72/.22/4,
    .70/1.20/.65/24,2.30/1.83/.65/115,1.95/.58/.40/78,
    1.75/2.18/.30/66,2.94/.68/.28/108}
    \draw[fieldarrow,orange!65!black] (\x,\y) -- ++(\angle:\len);
  \node at (1.75,-.34) {$s(u)\ \text{nonconstant}$};
  \node at (1.75,-.70) {Positive angular energy};
\end{scope}
\end{tikzpicture}
\caption{The equality case in the vector capacity problem.
Both panels have $f=2h_Ds$ with $|s|=1$, and hence the same modulus.
The shaded set $D=D_1\cup D_2$ has $h_D=1$; the potential is positive
in $\Lambda$ and vanishes on its boundary.
Arrows schematically depict internal vectors, not spatial flow.
Varying $s$ adds $4\int_\Lambda h_D^2|\nabla s|^2$ to the Dirichlet
energy. Connectedness of $\Lambda$ forces $s$ to be constant whenever
this term vanishes.}
\label{fig:orientation}
\end{figure}
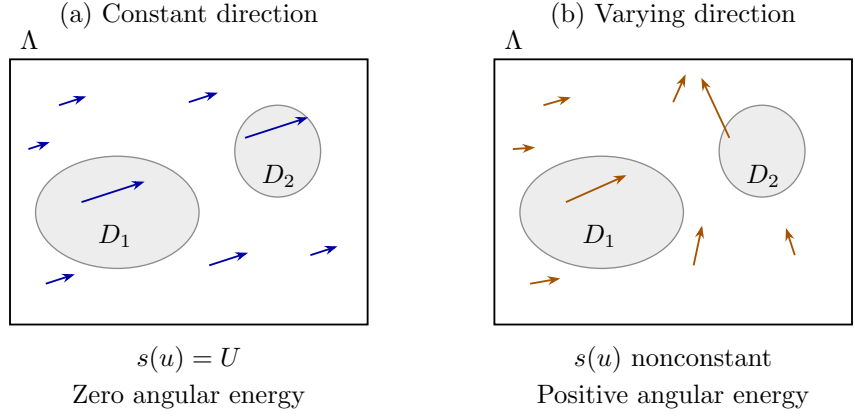

\begin{lemma}\label{lem:gap}
As $\eta\downarrow0$,
$\inf_{\{G\le\eta\}}I\to C_D$. Furthermore, for every $\delta>0$
there are $\eta,c>0$ such that
\[
 \inf\{I(f):G(f)\le\eta,\ \dist_{L^2}(f,\M)\ge\delta\}
                    \ge C_D+c.
\]
\end{lemma}
\begin{proof}
For the first claim, a counterexample would give a bounded-rate sequence
$f_j$ with $G(f_j)\to0$ and $I(f_j)\le C_D-\varepsilon$ for some
fixed $\varepsilon>0$.
Compactness of rate sublevels, continuity of $G$, and lower
semicontinuity of $I$ would contradict Lemma~\ref{lem:minimizers}.
For the second claim, failure gives a bounded-rate sequence with
$G(f_j)\to0$, distance at least $\delta$ from $\M$, and
$I(f_j)\le C_D+o(1)$. A convergent subsequence has a limit in
$\{G=0\}$ at minimum rate, still at distance at least $\delta$
from $\M$, again a contradiction.
\end{proof}

\section{Conditioning and macroscopic alignment}\label{sec:conditioning}

We use the $L^2$ large-deviation principle proved in
Proposition~\ref{prop:LDP}: the unconditioned $X_n$ have speed
$(\log n)^2$ and rate $I$.
By Lemma~\ref{lem:deficit}, for each fixed $\eta>0$,
\[
 \Prob(\Omega_n)(1-o(1))
       \le \Prob\bigl(G(X_n)\le\eta\bigr).
\]
The closed-set large-deviation upper bound and
Lemma~\ref{lem:gap}, followed by $\eta\downarrow0$, give the upper
bound in \eqref{eq:cost}. Proposition~\ref{prop:KS} gives the lower bound.

For \eqref{eq:shape}, fix $\delta>0$ and take $\eta,c$ from
Lemma~\ref{lem:gap}. Then
\begin{align*}
 Q_n\bigl(\dist_{L^2}(X_n,\M)\ge\delta\bigr)
 &\le Q_n(G(X_n)>\eta)\\
 &\quad+
 \frac{\Prob(G(X_n)\le\eta,\ \dist_{L^2}(X_n,\M)\ge\delta)}
      {\Prob(\Omega_n)}.
\end{align*}
The first term tends to zero. The upper bound for closed sets and
\eqref{eq:lower-cost} make the second at most
$\exp\{-(c+o(1))(\log n)^2\}$.
To identify the direction explicitly, set
$A_n=\int_\Lambda X_n(u)h_D(u)\,du$ and $U_n=A_n/|A_n|$.
The vector $A_n$ is nonzero almost surely under $Q_n$ by absolute
continuity of the finite-dimensional Gaussian law.
Rotation invariance makes $U_n$ exactly uniform on $S^{m-1}$.
Expanding the squared $L^2$ norm shows that $2U_nh_D$ is a nearest
point to $X_n$ in $\M$. Thus \eqref{eq:shape} gives
$\|X_n-2U_nh_D\|_{L^2}\to0$ in probability, while $2U_nh_D$
already has the claimed limiting law.
This proves Theorem~\ref{thm:main}.

\begin{proof}[Proof of Corollary~\ref{cor:spins}]
For $z\ne0$ and $e\in S^{m-1}$,
\begin{equation}\label{eq:normalize}
 \left|\frac z{|z|}-e\right|\le |z-2e|.
\end{equation}
Indeed, insert $z/2$ and use $\big||z|-2\big|\le|z-2e|$.
Since $h_D=1$ on $D$, \eqref{eq:shape}, \eqref{eq:normalize},
and the bounded boundary-cell error show that
\[
 \inf_{e\in S^{m-1}}\frac1{|D_n|}
            \sum_{x\in D_n}|\sigma_n(x)-e|^2
       \longrightarrow0
\]
in $Q_n$-probability and in mean. Boundedness justifies the latter
convergence.
Let
\[
 a_n=\frac1{|D_n|^2}\sum_{x,y\in D_n}
                \E_{Q_n}|\sigma_n(x)-\sigma_n(y)|^2.
\]
The preceding convergence implies $a_n\to0$; equivalently,
if $M_n=|D_n|^{-1}\sum_x\sigma_n(x)$, then
$a_n=2\E_{Q_n}(1-|M_n|^2)\to0$.
Choose a deterministic $x_n$ for which the average over $y$ of
$\E_{Q_n}|\sigma_n(x_n)-\sigma_n(y)|^2$ is at most $a_n$.
When $a_n>0$, let
\[
 \widetilde D_n=
 \{y\in D_n:\E_{Q_n}|\sigma_n(x_n)-\sigma_n(y)|^2\le\sqrt{a_n}\}.
\]
Its complement has relative size at most $\sqrt{a_n}$.
For $y,z\in\widetilde D_n$ the squared triangle inequality gives
$\E_{Q_n}|\sigma_n(y)-\sigma_n(z)|^2\le4\sqrt{a_n}$.
If $a_n=0$, take $\widetilde D_n=D_n$.
\end{proof}

\section{Gaussian large deviations with the required normalization}

The scalar droplet analysis of~\cite{BAD} uses an $L^2$
large-deviation principle. The following direct verification
records the normalization and the step-function topology used here.
The vector dimension is fixed throughout.

\begin{proposition}\label{prop:LDP}
The unconditioned laws of $X_n$ satisfy a good large-deviation principle
on $L^2(\Lambda;\R^m)$ with speed $(\log n)^2$ and rate \eqref{eq:rate}.
\end{proposition}
\begin{proof}
Translate the continuum box to $(0,1)^2$. Let $\ell\times\ell$ be the
number of lattice sites: $\ell=n$ for odd $n$, and $\ell=n+1$ for even $n$.
Put $h=1/(\ell+1)$ and $g=1/(2\pi)$.
On cells of side $h$ centered at $(ih,jh)$, embed the discrete sine
vectors as step functions $e_{\ell,p}$ with values
$2\sin(\pi p_1ih)\sin(\pi p_2jh)$, and set them to zero in the
remaining boundary strips. These functions are exactly orthonormal.
For each fixed $p\in\mathbb N^2$ they converge in $L^2$ to
$e_p(u)=2\sin(\pi p_1u_1)\sin(\pi p_2u_2)$.

The unscaled embedded field has the expansion
\[
 Y_\ell=\sum_{1\le p_1,p_2\le\ell}
          \sqrt{q_{\ell,p}}\,\xi_p e_{\ell,p},\qquad
 q_{\ell,p}=
 \frac1{4g(\ell+1)^2
       \sum_{j=1}^2\sin^2(\pi p_j/(2(\ell+1)))},
\]
where the $\xi_p$ are independent standard Gaussian vectors in $\R^m$.
Consequently,
\begin{equation}\label{eq:modes}
 q_{\ell,p}\longrightarrow q_p=\frac2{\pi|p|^2},\qquad
 q_{\ell,p}\le\frac{\pi}{2|p|^2},\qquad
 \sum_pq_{\ell,p}\le C\log(\ell+1).
\end{equation}
The elementary inequality $\sin t\ge2t/\pi$ on $[0,\pi/2]$
gives the middle bound; summation in dyadic annuli gives the last.

To obtain the literal cells defining $X_n$, apply the linear map $T_n$
that retains site values on the corresponding cells of side $1/n$,
cut off at $\partial\Lambda$. Thus, after translation,
$X_n=T_nY_\ell/\log n$.
Their areas are at most $n^{-2}$, so
$\|T_n\|^2\le(\ell+1)^2/n^2=1+O(n^{-1})$.
For fixed $p$, $T_ne_{\ell,p}\to e_p$ in $L^2$.
Let $Y_{n,>K}$ be the sum of the transformed modes
$\sqrt{q_{\ell,p}}\,\xi_pT_ne_{\ell,p}$ with $|p|>K$.
Equation~\eqref{eq:modes} therefore yields
\[
 \|\operatorname{Cov}(Y_{n,>K})\|\le CK^{-2},\qquad
 \operatorname{Tr}\operatorname{Cov}(Y_{n,>K})
                       \le C_m\log(n+1).
\]
For a centered finite-rank Hilbert-space Gaussian $Y$ with covariance
$Q$, diagonalization gives, when $0\le2t\|Q\|\le1/2$,
\[
 \log\E e^{t\|Y\|^2}
    =-\frac12\operatorname{Tr}\log(1-2tQ)
    \le2t\operatorname{Tr}Q.
\]
Taking $t$ a sufficiently small multiple of $K^2$ shows, for every
$r>0$,
\[
 \limsup_n\frac1{(\log n)^2}
 \log\Prob(\|Y_{n,>K}\|>r\log n)\le-c r^2K^2.
\]
For fixed $K$, the difference between the retained Gaussian sums,
obtained by replacing $\sqrt{q_{\ell,p}}\,T_ne_{\ell,p}$ with
$\sqrt{q_p}\,e_p$, has covariance operator tending to zero.
The same Gaussian estimate, with arbitrary fixed $t$ and then
$t\to\infty$, makes this replacement exponentially negligible.
Thus
\[
 Z_{n,K}=\frac1{\log n}\sum_{|p|\le K}\sqrt{q_p}\,\xi_pe_p
\]
is an exponentially good approximation to $X_n$ as $K\to\infty$.

Write $V_K=\operatorname{span}\{e_p:|p|\le K\}\otimes\R^m$.
Let $P_K$ be the $L^2$-orthogonal projection onto $V_K$, and write
$f_p=\int_{(0,1)^2}f(u)e_p(u)\,du\in\R^m$.
The finite-dimensional Gaussian principle gives rate
$I_K(f)=\frac12\sum_p|f_p|^2/q_p$ on $V_K$ and infinity off $V_K$.
The Fourier characterization of the Dirichlet energy gives
$I_K(P_Kf)=I(P_Kf)\uparrow I(f)$, with $I$ as in \eqref{eq:rate}.
For completeness, the approximation also gives the full bounds directly.
For an open set containing a finite-rate $f$, choose a ball around $f$,
then $K$ so large that its projection $P_Kf$ lies in that ball and
the approximation error has exponential rate strictly larger than
$I(f)$. Subtracting that error from the finite-dimensional lower
bound gives the open-set lower bound.
For closed $F$ and $\delta>0$,
\[
 \limsup_n\frac{\log\Prob(X_n\in F)}{(\log n)^2}
 \le\max\left\{-\inf_{F^\delta\cap V_K}I,\,-c\delta^2K^2\right\},
\]
where $F^\delta=\{f:\dist(f,F)\le\delta\}$.
Take $\delta=K^{-1/2}$ and then $K\to\infty$.
Compactness of rate sublevels gives
$\liminf_K\inf_{F^{K^{-1/2}}\cap V_K}I\ge\inf_F I$.
This proves the closed-set upper bound and completes the proof.
\end{proof}

\section{Further questions}

The theorem determines the leading conditional profile. Several finer
questions remain beyond its scope. First, can one obtain quantitative
rates of profile concentration and remove the exceptional set in
Corollary~\ref{cor:spins} on a fixed interior subdomain?
For a nonempty open $D'$ with $\overline{D'}\subset D$, this would mean
\[
 \inf_{x,y\in nD'\cap\Z^2}
 \E_{Q_n}[\sigma_n(x)\cdot\sigma_n(y)]\longrightarrow1.
\]
This asks for uniformity of the expected pairwise alignment; simultaneous
uniform alignment in a field realization is a different question.

Second, the residual $X_n-2U_nh_D$, with $U_n$ chosen by the projection
in Section~\ref{sec:conditioning}, suggests a fluctuation problem. What additional centering,
normalization and topology yield a nontrivial limit, and how do the
components parallel and perpendicular to $U_n$ differ?
The present $L^2$ convergence does not identify these finer scales.

Finally, the leading avoidance cost in \eqref{eq:cost} is independent
of the fixed radius $R$ and vector dimension $m$. Determining the next
asymptotic order could reveal their effect on the cost.
For a translated ball or an anisotropic obstacle, one can also ask
which orientations are selected and at what order the breaking of
rotational symmetry becomes visible.

\end{document}